\documentclass[a4paper,reqno,12pt]{amsart}
\usepackage[english]{babel}
\usepackage[T1]{fontenc}
\usepackage[utf8]{inputenc}
\usepackage{mathptmx,amsmath,amssymb,mathtools}
\usepackage{fullpage,microtype,enumitem}
\usepackage[hidelinks,hypertexnames=false]{hyperref}
\newtheorem{theorem}{Theorem}[section]
\newtheorem{definition}[theorem]{Definition}
\newtheorem{lemma}[theorem]{Lemma}
\newtheorem{proposition}[theorem]{Proposition}
\newtheorem{corollary}[theorem]{Corollary}
\newtheorem{example}[theorem]{Example}
\newcommand{\ddc}{dd^c}
\newcommand{\PSH}{\mathop{\rm PSH}\nolimits}
\newcommand{\Eloc}{\mathcal E_{\mathrm{loc}}}
\renewcommand{\Cap}{\operatorname{Cap}}
\newcommand{\np}{\mathrm{np}}
\newcommand{\pp}{\mathrm p}
\newcommand{\ind}{\mathbf 1}
\newcommand{\cc}{\Subset}
\begin{document}
\author{Thai Duong Do$^{1,2}$}
\address{$^{1}$Institute for Artificial Intelligence, VNU University of Engineering and Technology, Hanoi, Vietnam.}
\address{$^{2}$Department of Mathematics, National University of Singapore, 10 Lower Kent Ridge Road, 119076, Singapore.}
\email{dtduong@vnu.edu.vn, dtduong@nus.edu.sg}
\author{Pham Hoang Hiep$^{1}$}
\address{$^{1}$Institute for Artificial Intelligence, VNU University of Engineering and Technology, Hanoi, Vietnam.}
\email{phhiepvn@vnu.edu.vn}
\title{COMPARISON PRINCIPLES FOR LOCAL PLURISUBHARMONIC POTENTIALS\\ON COMPLEX MANIFOLDS AND APPLICATIONS}
\date{October 3, 2026}
\begin{abstract}
We prove a weighted comparison principle for local plurisubharmonic potentials in Cegrell's class on an arbitrary complex manifold under suitable boundary and gluing conditions. We obtain domination and uniqueness when the manifold admits a global psh function in $\mathcal E_{\mathrm{loc}}$ whose Monge-Amp\`ere measure has an absolutely continuous part with positive density almost everywhere. We apply these results to plurisubharmonic functions on domains in $\mathbb C^2$ whose singularities satisfy local inequalities involving $\alpha\log\|f\|$ and a Cegrell potential. For such functions, maximality is equivalent to the vanishing of the Monge-Amp\`ere measure of the current remaining after subtraction of the Siu divisorial part. Consequently, maximality is a local property in this class.
\end{abstract}
\maketitle
\tableofcontents

\section{Introduction and main results}

Comparison principles are basic tools in pluripotential theory and the study of complex Monge-Amp\`ere equations. They are used to prove domination and uniqueness and to obtain estimates for existence, stability and regularity of solutions \cite{Cegrell98,Cegrell04,Kol98,GZ07,KN19}.

The local theory rests on Bedford and Taylor's operator and capacity \cite{BT76,BT82}, Cegrell's classes and convergence results \cite{Cegrell98,Cegrell04,Cegrell08,Cegrell12}, and B\l ocki's characterizations of the domain of definition \cite{Blocki04,Blocki06}. Weighted comparison inequalities were developed by Xing \cite{Xing96} and Nguyen and Pham \cite{NP09}; for unbounded potentials, the mass on pluripolar sets requires separate treatment \cite{ACCP09,DP26}.

Applications include Ko\l odziej's $L^\infty$ estimates \cite{Kol98}, stability and H\"older regularity \cite{GZ12,DDGPKZ14}, and results on boundary measures, degenerate equations and geodesic rays \cite{Dem87,DemPali10,PS10}. Related existence and regularity results for complex Monge-Amp\`ere equations were established by Tosatti and Weinkove \cite{TW10} and Sz\'ekelyhidi and Tosatti \cite{ST11}. Comparison principles are also used to study prescribed singularities \cite{DDL18,TW12}, pluricomplex Green functions \cite[Remark 5.3]{MT19}, and equations in the Hermitian \cite{KN15,KN19,Nguyen16,LPT21} and parabolic \cite{Dang22,DoPham22} settings.

On compact K\"ahler manifolds, comparison principles with terms accounting for mass on pole sets were established by Le, Nguyen and Pham \cite[Theorem 6.1]{LNP07} and by Coman, Guedj and Zeriahi \cite[Theorem 1.9 and Definition 3.1]{CGZ08} for classes defined through local psh potentials. On bounded hyperconvex domains, \AA hag, Czy\.{z}, Lu and Rashkovskii proved domination and uniqueness in $\mathcal N(H)$ under a local order on singularities, without a weighted mass assumption \cite[Definition 3.1 and Theorem 3.5]{ACLR25}. The capacity condition in \cite{DP26} gives a different comparison of polar masses for globally defined $\mathcal E$-potentials.

Our aim is to extend the weighted inequalities of \cite{NP09,DP26} from domains in $\mathbb C^n$ to local $\mathcal E$-potentials on arbitrary complex manifolds. The currents need not have global psh potentials, but the two families must have common pluriharmonic transition functions and satisfy the boundary and gluing conditions below. The form of the inequality comes from \cite{NP09,DP26}. The main step here is the integration-by-parts formula of Lemma~\ref{lem:ibp} and the compact-support comparison of Proposition~\ref{prop:compact}: local truncations and energy estimates agree on overlaps, retain the full polar measures, and allow weights defined only near the support of the difference. The proof uses the local methods of \cite{NP09,Vu22}.

Domination and uniqueness follow if there exists $\rho\in\PSH(X)\cap\Eloc(X)$ with $(\ddc\rho)^n\ge f\,dV$ and $f>0$ almost everywhere. Lemma~\ref{lem:zeta} shows that $\zeta=e^\rho$ is locally bounded and satisfies the same positive-density condition; a local normalization of $\zeta$ then supplies the required weight. Neither continuity, an exhaustion property, nor finite total mass is assumed for $\rho$. Example~\ref{ex:blowup} illustrates the result on a manifold containing a compact curve and for currents without global potentials.

We also address local and global maximality \cite{Cegrell09,Blocki09}. Rashkovskii \cite{Rash98} studied maximal plurisubharmonic functions associated with holomorphic mappings. We prove a criterion for functions $u$ satisfying $\phi+\psi\le u$ and $u+\psi\le\phi$ locally, where $\phi=\alpha\log\|f\|$ and $\psi\in\mathcal E$. Such functions may have divisorial singularities and need not belong to $\Eloc$. After subtraction of the Siu divisorial part, the remaining current has local $\mathcal E$-potentials. Maximality is equivalent to the vanishing of its Monge-Amp\`ere measure and is therefore a local property in this class.

Let $X$ be a complex manifold of dimension $n$ with a countable atlas. We write $\PSH^-(X)$ for the nonpositive psh functions. A pluripolar set on $X$ means a locally pluripolar set. We use $d^c=(\mathrm i/2\pi)(\bar\partial-\partial)$ and $\ddc=(\mathrm i/\pi)\partial\bar\partial$. We fix a smooth positive volume form $dV$ on $X$. Its null sets agree with the Lebesgue null sets in coordinate charts.

A local psh function belongs to $\Eloc$ if, near each point, its restriction after subtraction of a constant belongs to Cegrell's class $\mathcal E$ on a smaller bounded hyperconvex coordinate neighborhood. On a bounded hyperconvex domain $\Omega$, the class $\mathcal E_0(\Omega)$ consists of bounded nonpositive psh functions with zero boundary values and finite Monge-Amp\`ere mass. The class $\mathcal F(\Omega)$ consists of decreasing limits of $\mathcal E_0(\Omega)$-functions with uniformly bounded masses, and $\mathcal E(\Omega)$ is its local counterpart; see \cite{Cegrell04}. All potentials are assumed nontrivial on each connected coordinate neighborhood. We pass to a countable locally finite refinement of the cover whenever necessary.

\begin{definition}\label{def:barrier}
Let $u=(u_j)$ be a family of local potentials on an open cover $(\Omega_j)$ such that $u_j\in\Eloc(\Omega_j)$ and
\[
u_j=u_l+h_{jl}\quad\text{on }\Omega_j\cap\Omega_l,
\]
where $h_{jl}$ are real pluriharmonic functions. A second family $v=(v_j)$ has the same transition functions if $v_j=v_l+h_{jl}$ on the overlaps.

We consider the following condition on $(u,v)$: there exist functions $h_j^k\in\PSH^-(\Omega_j)$ satisfying:
\begin{enumerate}[label=(\roman*),leftmargin=*]
\item $h_j^k\uparrow0$ almost everywhere as $k\to\infty$;
\item for every $k$ there is a compact $C_k\cc X$, independent of $j$, such that
\begin{equation}
u_j\ge v_j+h_j^k\quad\text{on }\Omega_j\setminus C_k;\label{eq:boundary}
\end{equation}
\item the functions $g_j^k=\max(u_j,v_j+h_j^k)$ satisfy
\begin{equation}
g_j^k=g_l^k+h_{jl}\quad\text{on }\Omega_j\cap\Omega_l.\label{eq:glue}
\end{equation}
\end{enumerate}
Almost everywhere refers to any smooth positive volume form. An equivalent formulation of (ii), for a compact exhaustion $K_l$ of $X$, is: for every $k$ there is $l(k)$ such that \eqref{eq:boundary} holds with $C_k=K_{l(k)}$.
\end{definition}

The differences $u_j-v_j$ agree on overlaps as locally integrable functions. Thus the two currents differ by a globally $\ddc$-exact current; in particular, they represent the same Bott--Chern class. The common-transition assumption restricts the pairs of currents covered by our results.

Theorem~\ref{th:comparison} requires all three conditions in Definition~\ref{def:barrier}. Condition (iii) is an additional hypothesis. It holds when $h_j^k=h^k|_{\Omega_j}$ for global psh functions $h^k$. Replacing $C_k$ by $\bigcup_{l=1}^{k}C_l$, we may assume that $C_k\subset C_{k+1}$. Choose open sets $D_k\cc X$ such that $C_k\subset D_k$, $D_k\cc D_{k+1}$ and $\bigcup_{k=1}^{\infty}D_k=X$.

The following formulas define global positive closed currents and their Cegrell products:
\[
T=\ddc u_j,\qquad S=\ddc v_j,\qquad R=\ddc\max(u_j,v_j).
\]
Their independence of the local representatives is proved in Lemma~\ref{lem:classes}. The sets $A=\{u<v\}$ and $E=\{u=v=-\infty\}$ are defined locally by
\[
A\cap\Omega_j=\{z\in\Omega_j:u_j(z)<v_j(z)\},\qquad
E\cap\Omega_j=\{z\in\Omega_j:u_j(z)=v_j(z)=-\infty\}.
\]
For each $j$, define $(v_j-u_j)_+:\Omega_j\to[0,+\infty]$ by
\[
(v_j-u_j)_+(z)=
\begin{cases}
\max\{v_j(z)-u_j(z),0\},& u_j(z),v_j(z)>-\infty,\\
+\infty,& u_j(z)=-\infty<v_j(z),\\
0,& v_j(z)=-\infty,
\end{cases}
\qquad z\in\Omega_j.
\]
On $\Omega_j\cap\Omega_l$, we have $u_j=u_l+h_{jl}$ and $v_j=v_l+h_{jl}$, with $h_{jl}$ finite and pluriharmonic. Thus the pole sets of the corresponding representatives coincide, and $v_j-u_j=v_l-u_l$ wherever both potentials are finite. The above conventions therefore give
\[
(v_j-u_j)_+=(v_l-u_l)_+\quad\text{on }\Omega_j\cap\Omega_l.
\]
These local Borel functions define a global nonnegative Borel function $(v-u)_+$ on $X$ by
\[
(v-u)_+\big|_{\Omega_j}=(v_j-u_j)_+.
\]
The same transition identities show that $A$ and $E$ are well defined. In particular, $(v-u)_+=0$ on the common pole set $E$; no subtraction of $-\infty$ from $-\infty$ is used. The measures used in integrals involving $(v-u)_+$ vanish on pluripolar sets. Thus the values assigned on the pole sets do not affect these integrals. The Monge-Amp\`ere mass on $E$ appears separately in the comparison inequality below.

For negative $u,v\in\mathcal E(\Omega)$ on a bounded hyperconvex coordinate domain $\Omega$, write $u\preceq_{\Cap}v$ if
\begin{equation}
\varliminf_{t\to\infty}t^n\Cap_\Omega\bigl(D\cap\{u<v-t\}\bigr)=0\quad\text{for every }D\cc \Omega.\label{eq:cap}
\end{equation}
For families of local potentials, near each point we choose a bounded hyperconvex coordinate neighborhood $\Omega\subset\Omega_j$ and a constant $C$ such that $u_j-C,v_j-C\in\mathcal E(\Omega)$. We require \eqref{eq:cap} for this pair on $\Omega$. Here the Bedford-Taylor capacity is
\[
\Cap_\Omega(E)=\sup\left\{\int_E(\ddc \phi)^n:
 \phi\in\PSH(\Omega),\ -1\le \phi\le0\right\}.
\] If $h$ is pluriharmonic, then
\[
\{u+h<v+h-t\}=\{u<v-t\}\qquad(t>0).
\]
Thus adding the same pluriharmonic function to $u$ and $v$ does not change the sets in \eqref{eq:cap}.

\begin{theorem}\label{th:comparison}
Let $X$ be an arbitrary complex manifold. Let $u,v$ be families of local $\mathcal E$-potentials with the same transition functions, satisfying conditions (i)--(iii) of Definition~\ref{def:barrier}. For $\phi_1,\ldots,\phi_n\in\PSH(X)$ with $-1\le \phi_l\le0$, put $\mu_\phi=\ddc \phi_1\wedge\cdots\wedge\ddc \phi_n$. Then
\begin{equation}\begin{aligned}
\frac1{n!}\int_A(v-u)^n\mu_\phi+\int_A(-\phi_1)S^n+\int_E(-\phi_1)R^n
\le\int_A(-\phi_1)T^n+\int_E(-\phi_1)T^n.
\end{aligned}\label{eq:full}
\end{equation}
If $u\preceq_{\Cap}v$ locally, then
\begin{equation}
\frac1{n!}\int_A(v-u)^n\mu_\phi+\int_A(-\phi_1)S^n
\le\int_A(-\phi_1)T^n.\label{eq:clean}
\end{equation}
All nonnegative integrals are extended integrals.
\end{theorem}

For domination and uniqueness, we assume that there exist $\rho\in\PSH(X)\cap\Eloc(X)$ and $f\in L^1_{\mathrm{loc}}(X,dV)$ such that
\begin{equation}
(\ddc\rho)^n\ge f\,dV,\qquad f>0\quad dV\text{-almost everywhere}.
\label{eq:rho}
\end{equation}
The measure in \eqref{eq:rho} is the full local Cegrell measure. The existence of such a positive density does not depend on the choice of $dV$. In particular, it holds if $(\ddc\rho)^n=f\,dV$ with $f>0$ almost everywhere. It also holds whenever $X$ admits a continuous strictly psh function $\rho$. Indeed, $\rho\in\Eloc(X)$, and on each sufficiently small coordinate neighborhood there is $\epsilon>0$ such that $\rho-\epsilon|z|^2$ is psh. Positivity and multilinearity of bounded Bedford--Taylor products give
\[
(\ddc\rho)^n\ge\epsilon^n(\ddc|z|^2)^n.
\]
Thus the absolutely continuous part of $(\ddc\rho)^n$ has positive density almost everywhere with respect to $dV$. No finite total mass assumption on $(\ddc\rho)^n$ is required.

On a connected compact manifold, every global psh function is constant. Hence the first term in \eqref{eq:full} vanishes there, and \eqref{eq:rho} cannot hold. The domination and uniqueness results below apply to noncompact manifolds satisfying \eqref{eq:rho}. This condition allows compact complex curves, as Example~\ref{ex:blowup} shows.

\begin{theorem}\label{th:dom}
Suppose that $X$ satisfies \eqref{eq:rho}. Let $u,v$ satisfy the hypotheses of Theorem~\ref{th:comparison}. Assume that $T^n\le S^n$ on $A$ and
\begin{equation}
(T^n)_{\pp}=(R^n)_{\pp}.\label{eq:balance}
\end{equation}
Then $u_j\ge v_j$ on every $\Omega_j$. In particular, \eqref{eq:balance} may be replaced by the local condition $u\preceq_{\Cap}v$. No finite global mass assumption is needed.
\end{theorem}

\begin{theorem}\label{th:unique}
Suppose that $X$ satisfies \eqref{eq:rho}. Let $u,v$ be families of local $\mathcal E$-potentials with the same transition functions, and suppose that $T^n=S^n=\mu$. Assume that the condition in Definition~\ref{def:barrier} holds for both $(u,v)$ and $(v,u)$. If
\[
(R^n)_{\pp}=\mu_{\pp},
\]
then $u=v$. It suffices that every point have a coordinate neighborhood on which either $u\preceq_{\Cap}v$ or $v\preceq_{\Cap}u$; the direction may vary from one neighborhood to another.
\end{theorem}

For the application, let $u\in\PSH(\Omega)$, where $\Omega\subset\mathbb C^2$. We assume that near every point there are a bounded hyperconvex coordinate neighborhood $D$, a function $\psi\in\mathcal E(D)$ and a function $\phi=\alpha\log\|f\|$, where $\alpha>0$ and $f$ is a holomorphic mapping not identically zero, such that
\begin{equation}
\phi+\psi\le u,\qquad u+\psi\le \phi.\label{eq:CD}
\end{equation}
The choice $\phi=0$ is also allowed. The inequalities in \eqref{eq:CD} are understood as inequalities of psh functions, including their extended values.

Write the Siu decomposition of $\ddc u$ as
\[
\ddc u=S_u+T_u,\qquad S_u=\sum_{j=1}^{N_u}\nu_{Z_j}(u)[Z_j],
\]
where $Z_j$, $1\le j\le N_u$, are the distinct irreducible analytic hypersurfaces with positive generic Lelong number, $N_u\in\{0,1,2,\ldots\}\cup\{\infty\}$, and $[Z_j]$ is the current of integration over $Z_j$. If $N_u=0$, the sum is zero. The current $T_u$ is the positive closed residual current. For functions satisfying \eqref{eq:CD} locally in dimension two, the sum defining $S_u$ is locally finite and $T_u$ has local $\mathcal E$-potentials, as established below.

\begin{theorem}\label{th:CD}
Let $\Omega\subset\mathbb C^2$ be a domain and let $u\in\PSH(\Omega)$ satisfy \eqref{eq:CD} locally. Choose local psh potentials $h_j$ of $S_u$ on coordinate neighborhoods $D_j$, and let $f_j$ be the psh representative of the distribution $u-h_j$. Then $f_j\in\Eloc(D_j)$, and the local Cegrell measures
\begin{equation}
T_u^2\big|_{D_j}:=(\ddc f_j)^2
\label{eq:residualMA}
\end{equation}
define a positive Radon measure on $\Omega$, independent of the cover and the potentials $h_j$. The following conditions are equivalent:
\begin{enumerate}[label=(\roman*),leftmargin=*]
\item $u$ is maximal on $\Omega$;
\item $u$ is locally maximal on $\Omega$;
\item $T_u^2=0$ on $\Omega$, or equivalently $(\ddc f_j)^2=0$ on every $D_j$.
\end{enumerate}
\end{theorem}

The measure in \eqref{eq:residualMA} is the full Cegrell measure, including any mass on pluripolar sets. The criterion can be written locally as
\[
\left(\ddc\left(u-\sum_{l=1}^{N}\alpha_l\log|g_l|\right)\right)^2=0,
\qquad S_u=\sum_{l=1}^{N}\alpha_l[\{g_l=0\}],
\]
where the $g_l$ are reduced local defining functions for the irreducible divisorial components and the expression in parentheses denotes its psh representative. The domain in Theorem~\ref{th:CD} need not be pseudoconvex. The proof is given in Section~\ref{sec:CD}.

\section{Preliminaries}
For a local Cegrell measure $\mu$, write $\mu=\mu_{\np}+\mu_{\pp}$, where $\mu_{\np}$ vanishes on pluripolar sets and $\mu_{\pp}$ is carried by a pluripolar set. For $\mu=(\ddc u)^n$, the decomposition theorem gives
\[
\mu_{\pp}=\ind_{\{u=-\infty\}}\mu,\qquad
\mu_{\np}=\ind_{\{u>-\infty\}}\mu.
\]
For $u\in\mathcal F(\Omega)$, these identities follow from \cite[Theorem 5.11]{Cegrell04}. For $u\in\mathcal E(\Omega)$ and $D\cc\Omega$, choose a function in $\mathcal F(\Omega)$ equal to $u$ on a neighborhood of $\overline D$, using the second remark after Definition 4.6 in \cite{Cegrell04}. Ordinary locality then gives the identities on $D$. This proves the assertion for $\Eloc$ without a global mass assumption. The decompositions agree on overlaps and define global measures. We recall the following results:
\begin{enumerate}[label=(\alph*),leftmargin=*]
\item Mixed Monge-Amp\`ere products are defined by decreasing $\mathcal E_0$-approximation in \cite[Theorem 4.2]{Cegrell04}. Local $\mathcal F$-representatives are provided by the second remark after Definition 4.6 there. Convergence with a bounded psh weight for sequences in $\mathcal E$ lying above their limits is proved in \cite[Lemma 3.2]{Cegrell08}. Lemma~\ref{lem:weighted} below records the local form used here.
\item If $u\le v$ are local $\mathcal E$-potentials, then
\[
\bigl((\ddc v)^n\bigr)_{\pp}\le\bigl((\ddc u)^n\bigr)_{\pp}.
\]
This is the Cegrell-class form of Demailly's comparison of residual Monge-Amp\`ere measures \cite[Remark 3.6]{Dem87}. The required statement and its mixed version are proved in \cite[Lemma 4.1 and the first following remark]{ACCP09}. We also use the preceding identification of the polar part.
\item The scalar comparison on pluripolar sets, in the liminf form explained below, gives, for local $u\preceq_{\Cap}v$,
\[
\bigl((\ddc u)^n\bigr)_{\pp}\le\bigl((\ddc v)^n\bigr)_{\pp}.
\]

\item Plurifine locality of bounded Bedford--Taylor products relative to a closed positive current is proved in \cite[Theorem 2.9]{Vu21}. Locality for maxima in Cegrell's classes is proved in \cite[Theorem 4.1]{NP09}. The bounded integration-by-parts formulas and Cauchy--Schwarz inequality are due to Bedford and Taylor \cite[Theorem 2.7, Corollaries 2.8--2.9 and Proposition 3.3]{BT82}. The weighted comparison estimates with a fixed closed positive current are proved in \cite[Lemmas 3.2--3.3]{NP09}.
\end{enumerate}

For (c), we use the scalar comparison in \cite[Lemmas 3.2--3.4]{DP26}. We recall why the liminf condition in \eqref{eq:cap} suffices. The estimate for $u,v\in\mathcal F$ has the form
\[
(1-\epsilon)^n\int_{\{u<-2t/\epsilon\}}(\ddc u)^n
\le \int_{\{v<-t\}}(\ddc v)^n
+\left(\frac2\epsilon-1\right)^n t^n
\Cap_\Omega\{u\le v-t\}.
\]
The estimate holds for $0<\epsilon<1$ and $t>0$. Letting $t$ tend to infinity along a sequence making the last term vanish, and then letting $\epsilon\downarrow0$, gives the inequality between the total polar masses. To pass from strict to non-strict sublevel sets, put $E=\{u=v=-\infty\}$. For every $t>0$,
\[
\{u\le v-2t\}\subset\{u<v-t\}\cup E.
\]
Since $E$ is pluripolar and has zero Bedford--Taylor capacity, for every $D\cc\Omega$ we have
\[
\Cap_\Omega\bigl(D\cap\{u\le v-2t\}\bigr)
\le\Cap_\Omega\bigl(D\cap\{u<v-t\}\bigr).
\]
Thus, if $(t_k)$ is a sequence for which the expression in \eqref{eq:cap} tends to zero, the corresponding non-strict capacity term tends to zero along $(2t_k)$.

We give the localization argument. Let $u,v\in\mathcal E(\Omega)$ satisfy \eqref{eq:cap}, and let $E\cc\Omega$ be compact and pluripolar. Choose $E\cc D_0\cc D\cc\Omega$. Lemma 4.3 of \cite{ACCP09} gives a function $f\ge u$ such that
\[
(\ddc f)^n=\ind_E(\ddc u)^n.
\]
In that construction, $f\ge u_{\ind_{D_0}}$, in the notation of Definition 4.2 of \cite{ACCP09}. Property (4) following that definition gives $u_{\ind_{D_0}}\in\mathcal F(\Omega)$ because $D_0\cc\Omega$. Thus $f\in\mathcal F(\Omega)$ by Lemma~\ref{lem:classes} below. Set
\[
h=\left(\sup\{\psi\in\PSH^-(\Omega):\psi\le f
\text{ on }\Omega\setminus\overline{D_0}\}\right)^*.
\]
This is the envelope construction of \cite[Section 2]{Cegrell08}. Since $f\le h\le0$, Lemma~\ref{lem:classes} gives $h\in\mathcal F(\Omega)$. The definition gives $h=f$ on $\Omega\setminus\overline{D_0}$. The envelope is maximal on $D_0$: a psh function above $h$ and equal to $h$ outside a compact subset of $D_0$ is still admissible in the defining supremum. Hence $(\ddc h)^n=0$ on $D_0$ by \cite[Proposition 2.2]{Blocki04}, which is stated in arbitrary dimension for the domain of definition of the Monge-Amp\`ere operator. By (b), the polar part of $(\ddc h)^n$ is dominated by $(\ddc f)^n$, which is supported in $E\subset D_0$. Thus $h\in\mathcal F^a(\Omega)$, where $\mathcal F^a$ denotes the functions in $\mathcal F$ whose Monge-Amp\`ere measures vanish on pluripolar sets. Proposition 3.4 of \cite{CKZ05} gives
\[
\lim_{t\to\infty}t^n\Cap_\Omega\{h<-t\}=0.
\]
Put $g=\max(f,v)$. Since $f\le g\le0$, we have $g\in\mathcal F(\Omega)$ by Lemma~\ref{lem:classes}. For $t>0$,
\[
\{f<g-t\}\subset\bigl(D\cap\{u<v-t\}\bigr)\cup\{h<-t\}.
\]
Indeed, on the set on the left we have $g=v$ and $u\le f<v-t$; outside $D$ we also have $f=h$ and $g\le0$. Therefore
\[
t^n\Cap_\Omega\{f<g-t\}
\le t^n\Cap_\Omega\bigl(D\cap\{u<v-t\}\bigr)
+t^n\Cap_\Omega\{h<-t\}.
\]
Choose $t_k\to\infty$ so that the first term on the right tends to zero, using \eqref{eq:cap}. The second term tends to zero along every such sequence. The strict-to-nonstrict inclusion above then gives
\[
(2t_k)^n\Cap_\Omega\{f\le g-2t_k\}\longrightarrow0.
\]
Apply the scalar $\mathcal F$ estimate to $f,g$ along $2t_k$, and then let $\epsilon\downarrow0$. Since $g\ge f$ and $g\ge v$, (b) shows that the polar part of $(\ddc g)^n$ is supported in $E$ and dominated by that of $(\ddc v)^n$. Consequently,
\[
\int_E(\ddc u)^n
=\int_\Omega\bigl((\ddc f)^n\bigr)_{\pp}
\le\int_\Omega\bigl((\ddc g)^n\bigr)_{\pp}
\le\int_E(\ddc v)^n.
\]
Inner regularity extends this inequality to every pluripolar Borel set and proves (c).

\begin{lemma}\label{lem:classes}
Let $\Omega$ be a bounded hyperconvex domain. If $f\in\mathcal F(\Omega)$ and $f\le g\le0$ with $g\in\PSH(\Omega)$, then $g\in\mathcal F(\Omega)$ and
\[
\int_\Omega(\ddc g)^n\le\int_\Omega(\ddc f)^n.
\]
If $u\in\mathcal E$, $g\in\PSH^-$ and $u\le g$, then $g\in\mathcal E$. The class $\Eloc$ is stable under addition of pluriharmonic functions, and local Cegrell products are unchanged by such additions.
\end{lemma}
\begin{proof}
Choose $f_k\in\mathcal E_0(\Omega)$ decreasing to $f$ with uniformly bounded masses. The functions $g_k=\max(f_k,g)$ are bounded, have zero boundary values, and have masses at most those of $f_k$ by the bounded comparison principle. Thus $g_k\in\mathcal E_0(\Omega)$, $g_k\downarrow g$, and $g\in\mathcal F(\Omega)$. Proposition 5.1 of \cite{Cegrell04}, with weight $-1$, applied to both sequences yields the displayed mass inequality.

If $u\in\mathcal E(\Omega)$ and $u\le g\le0$, choose, for each $D\cc \Omega$, a function $f\in\mathcal F(\Omega)$ equal to $u$ on a neighborhood of $\overline D$. The function $\max(f,g)$ belongs to $\mathcal F(\Omega)$ by the first assertion and equals $g$ on $D$. This proves the assertion for $\mathcal E$.

For invariance, let $u\in\Eloc$ and let $h$ be pluriharmonic. On a smaller coordinate ball, $h$ is bounded and $u+h\ge u-C$ for some constant $C$. After normalization by constants, the preceding assertion shows that $u+h\in\Eloc$. Here subtraction of a positive constant preserves the local class, since constants are locally bounded psh functions and $\mathcal E$ is a convex cone. Choose bounded psh $u_k\downarrow u$ on the ball. Then $u_k+h\downarrow u+h$, and both sequences are locally bounded. Their Bedford-Taylor currents agree at every level because $\ddc h=0$. Decreasing mixed convergence in $\mathcal E$ gives the equality of their limiting products. For several factors, apply the same argument simultaneously to each factor. Fixed constants can be subtracted on a smaller ball to make every approximant nonpositive. This also proves that the local products agree on overlaps and define global currents.
\end{proof}

\begin{lemma}\label{lem:weighted}
Let $D\cc \Omega$ be open, with $\Omega$ a bounded hyperconvex coordinate domain. For $1\le j\le n$, let $u_k^j\in\PSH^-(\Omega)$ decrease to $u^j\in\mathcal E(\Omega)$. If $h$ is bounded psh on $\Omega$, then
\[
h\,\ddc u_k^1\wedge\cdots\wedge\ddc u_k^n
\longrightarrow h\,\ddc u^1\wedge\cdots\wedge\ddc u^n
\]
weakly on $D$, tested against compactly supported continuous functions.
\end{lemma}
\begin{proof}
Lemma~\ref{lem:classes} gives $u_k^j\in\mathcal E(\Omega)$ for every $j,k$. The sequences decrease to nontrivial psh functions, so they converge in $L^1_{\mathrm{loc}}$ and hence weakly. For $h\le0$, the asserted convergence is \cite[Lemma 3.2]{Cegrell08}: its hypotheses are precisely $u_k^j,u^j\in\mathcal E(\Omega)$, $u_k^j\ge u^j$, weak convergence of each sequence, and bounded negative psh $h$. For a general bounded psh $h$, choose a constant $C>\sup_\Omega h$ and apply that lemma with weights $h-C$ and $-1$. Linearity gives the assertion for $h$. Testing against functions compactly supported in $D$ proves the stated local form.
\end{proof}

Finitely many smaller hyperconvex coordinate neighborhoods and a smooth partition of unity give this convergence near any fixed compact subset of $X$. On each such neighborhood a single constant makes all terms of a decreasing sequence nonpositive, since its first term is locally bounded above. Applying the lemma with the partition functions as continuous test functions gives convergence with a smooth cutoff and a bounded psh weight.

\begin{lemma}\label{lem:max}
Let $u,v\in\Eloc$ and $g=\max(u,v)$. Then
\[
\bigl((\ddc g)^n\bigr)_{\np}\ge
\ind_{\{u\ge v\}}\bigl((\ddc u)^n\bigr)_{\np}
+\ind_{\{u<v\}}\bigl((\ddc v)^n\bigr)_{\np}.
\]
If $u\preceq_{\Cap}v$ and $h\in\Eloc$ satisfies $u\le h\le g$, then the polar parts of $(\ddc u)^n$ and $(\ddc h)^n$ coincide.
\end{lemma}
\begin{proof}
Plurifine locality gives equality with $(\ddc v)^n$ on $\{u<v\}$. For the set $A_0=\{u\ge v\}\cap\{u>-\infty\}$, put $g_\epsilon=\max(u+\epsilon,v)$. Locality gives
\[
(\ddc g_\epsilon)^n\ge\ind_{A_0}(\ddc u)^n.
\]
Test this inequality against nonnegative continuous compactly supported functions and let $\epsilon\downarrow0$. Decreasing convergence yields the same domination for $(\ddc g)^n$. Taking nonpolar parts and restricting the resulting measure inequality to $\{u\ge v\}$ gives the assertion there, since the nonpolar part of $(\ddc u)^n$ gives no mass to $\{u=-\infty\}$. Combining the two disjoint pieces proves the first assertion.

For the second assertion, $\{u<h-t\}\subset\{u<v-t\}$ and $\{h<u-t\}=\varnothing$ for $t>0$. Thus $u\preceq_{\Cap}h$ and $h\preceq_{\Cap}u$. Apply (c) above in both directions.
\end{proof}

\section{Integration by parts and a comparison inequality}

\begin{lemma}\label{lem:mixedpolar}
Let $f_1,\ldots,f_{n-1},u,v\in\Eloc$ and set $T_0=\ddc f_1\wedge\cdots\wedge\ddc f_{n-1}$. The trace measure of $T_0$ vanishes on pluripolar sets. If $u\le v\le u+M$ for some constant $M$, then
\[
(\ddc u\wedge T_0)_{\pp}=(\ddc v\wedge T_0)_{\pp}.
\]
Moreover, the polar part of a mixed product is carried by the intersection of the pole sets of its factors.
\end{lemma}
\begin{proof}
Work on a smaller hyperconvex coordinate neighborhood, with all potentials normalized to be nonpositive. For a pluripolar Borel set $E$, the mixed inequality of \cite[Lemma 4.4]{ACCP09} gives
\[
\int_E\ddc g_1\wedge\cdots\wedge\ddc g_n
\le\prod_{l=1}^n\left(\int_E(\ddc g_l)^n\right)^{1/n}.
\]
Taking one factor equal to a smooth strictly psh function proves the assertion about the trace of $T_0$. Taking $E$ in the finite locus of one of the $g_l$ proves the assertion about the support of the polar part. Finally, apply the mixed polar monotonicity in \cite[Lemma 4.1 and the first following remark]{ACCP09} to $u\le v$ and $v\le u+M$, after a common constant shift making these functions nonpositive. Constants do not change the currents. For $n=1$, the empty product $T_0=1$ has its usual meaning.
\end{proof}

\begin{lemma}\label{lem:ibp}
Let $u\le v\le u+M$ be families of local $\mathcal E$-potentials with the same transition functions, equal outside a compact set $C$. Put $h=v-u$ off the common pole set and $h=0$ on that set. Write $T=\ddc u_j$, $S=\ddc v_j$. Let $T_0$ be a closed positive $(n-1,n-1)$-current locally given by a mixed product of local $\mathcal E$-potentials. If $\phi$ is bounded psh on a neighborhood of $C$, then
\begin{equation}
\int h\ddc \phi\wedge T_0
=\int\psi \phi(S-T)\wedge T_0,\label{eq:ibp1}
\end{equation}
where $\psi$ is a smooth compactly supported function in the neighborhood on which $\phi$ is defined and is equal to one near $C$. There is a positive locally finite measure $\nu$, independent of $k$, supported in $C$ and vanishing on pluripolar sets, such that, for every integer $k\ge2$,
\begin{equation}
\int h^k\ddc \phi\wedge T_0
=k\int \phi h^{k-1}(S-T)\wedge T_0
+k(k-1)\int \phi h^{k-2}\nu.\label{eq:ibpq}
\end{equation}
All integrals in these formulas are finite. The difference of the mixed measures in \eqref{eq:ibp1} has no polar part.
\end{lemma}
\begin{proof}
We first recall the formulas for bounded psh functions used below. On a coordinate neighborhood, for bounded psh functions $f,g$ and $h=f-g$, Bedford--Taylor calculus relative to the closed positive current $T_0$ defines the positive locally finite energy measure
\[
\mu_{T_0}(h)=dh\wedge d^ch\wedge T_0
=\frac12\ddc(h^2T_0)-h(\ddc f-\ddc g)\wedge T_0.
\]
For mixed products of bounded psh functions, the product rule, integration by parts and Cauchy--Schwarz inequality follow from \cite[Theorem 2.7, Corollaries 2.8--2.9 and Proposition 3.3]{BT82}. For $T_0$ with local $\mathcal E$-potentials, approximate its factors by decreasing sequences of bounded psh functions. Lemma~\ref{lem:weighted} gives convergence of the terms in the displayed identity and of their integrals with bounded psh weights. Polynomial weights in bounded psh functions are finite linear combinations of bounded psh functions, after addition of constants and polarization. Thus the same formulas and inequalities hold for $T_0$. Their plurifine locality follows by polarization from the locality of bounded Bedford--Taylor products relative to $T_0$, as established in \cite[Theorem 2.9]{Vu21} on an arbitrary coordinate domain. The truncation argument below also uses the energy estimates in \cite[Section 2]{Vu22}. We prove the required integration-by-parts formula on $X$ using local cutoffs.

In the present setting these energy measures also vanish on pluripolar sets. Indeed, after adding a constant to a bounded psh function $f$, we may suppose that $f\ge0$. Then $f^2$ is bounded psh and
\[
\mu_{T_0}(f)=\frac12\ddc(f^2)\wedge T_0-f\,\ddc f\wedge T_0.
\]
Both terms on the right vanish on pluripolar sets by Lemma~\ref{lem:mixedpolar}. The identity $\mu_{T_0}(f-g)=2\mu_{T_0}(f)+2\mu_{T_0}(g)-\mu_{T_0}(f+g)$ gives the same conclusion for bounded differences. These energy measures are unchanged by addition of constants.

We now give the local estimates that justify integration by parts. Fix a coordinate neighborhood and normalize the representatives to $u,v\le0$. Set
\[
u_l=\max(u,-l),\quad v_l=\max(v,-l),\quad
h_l=v_l-u_l,\quad g_l=\max(u+v,-l),\quad
\psi_l=1+\frac{g_l}{l}.
\]
Then $0\le h_l\le M$, $0\le\psi_l\le1$, and
\begin{equation}
\{\psi_l>0\}\subset\{u>-l,\ v>-l\}.\label{eq:cutofflocus}
\end{equation}
On every relatively compact subdomain, the masses of
\begin{equation}
\ddc u_l\wedge T_0,\qquad \ddc v_l\wedge T_0,\qquad
\ddc g_l\wedge T_0\label{eq:localmass}
\end{equation}
are bounded uniformly in $l$. Indeed, $u_l\downarrow u$, $v_l\downarrow v$, and $g_l\downarrow u+v$ in $\Eloc$; apply decreasing mixed convergence with the factors of $T_0$ fixed.

Let $f\ge0$ be a smooth compactly supported cutoff. Bounded Bedford--Taylor calculus gives
\[
\int f\,dh_l\wedge d^ch_l\wedge T_0
=\frac12\int h_l^2\ddc f\wedge T_0
-\int fh_l(\ddc v_l-\ddc u_l)\wedge T_0.
\]
Consequently,
\begin{equation}
\sup_l\int f\,dh_l\wedge d^ch_l\wedge T_0<\infty.\label{eq:gapenergy}
\end{equation}
The same identity applied to $\psi_l-1$ yields
\begin{align}
\int f\,d\psi_l\wedge d^c\psi_l\wedge T_0
&=\frac12\int(\psi_l-1)^2\ddc f\wedge T_0
-\frac1l\int f(\psi_l-1)\ddc g_l\wedge T_0
\longrightarrow0.\label{eq:cutoffenergy}
\end{align}
For the first term, use dominated convergence and Lemma~\ref{lem:mixedpolar}, since $\psi_l\uparrow1$ off $\{u+v=-\infty\}$. The second term is $O(l^{-1})$ by \eqref{eq:localmass}. Also $\ddc\psi_l\wedge T_0$ has mass $O(l^{-1})$ on compact subsets.

Write $\mu=(S-T)\wedge T_0$. By Lemma~\ref{lem:mixedpolar}, its polar part vanishes. On the set in \eqref{eq:cutofflocus}, plurifine locality identifies the truncated mixed products with the original ones, and $h_l=h$. Since $\psi_l=0$ on the complement, for every integer $k\ge1$ we obtain
\begin{equation}\begin{aligned}
\psi_l h_l^{k-1}(\ddc v_l-\ddc u_l)\wedge T_0
&=\psi_l h^{k-1}\mu\\
&\longrightarrow h^{k-1}\mu.
\end{aligned}\label{eq:cutoffmeasure}\end{equation}
The convergence is in total variation on compact subsets: the bounded factors converge $|\mu|$-almost everywhere, so dominated convergence applies to the fixed locally finite measure $|\mu|$. The common pole set is $|\mu|$-null; zeroth powers are interpreted as one.

The energy measures agree on the sets where the truncated potentials agree. If $j\ge l$, then $h_j=h_l$ on $\{\psi_l>0\}$, and bounded plurifine locality gives equality of their energy measures there. Since $\psi_j\ge\psi_l$ and the energy measures are positive, the measures
\[
\psi_l\,dh_l\wedge d^ch_l\wedge T_0
\]
increase to a positive measure $\nu$. It is locally finite by \eqref{eq:gapenergy}. Each term vanishes on pluripolar sets by the energy calculation at the beginning of the proof; hence so does the limit. Moreover, on $\{\psi_l>0\}$ we have $h_l=h$, so the energy terms multiplied by $h_l^{k-2}$ converge in total variation on compact subsets to $h^{k-2}\nu$ for every $k\ge2$.

Expand $\ddc(\psi_l h_l^k T_0)$. The terms containing $\ddc\psi_l$ tend to zero by \eqref{eq:localmass}. Those containing $d\psi_l$ tend to zero by Cauchy--Schwarz, \eqref{eq:gapenergy}, and \eqref{eq:cutoffenergy}. Using \eqref{eq:cutoffmeasure} for the remaining terms gives, as currents,
\begin{align}
\ddc(h T_0)&=\mu,\label{eq:localchain1}\\
\ddc(h^k T_0)&=k h^{k-1}\mu
+k(k-1)h^{k-2}\nu,\qquad k\ge2.\label{eq:localchainq}
\end{align}
The left sides converge distributionally because $\psi_l h_l^k\to h^k$ in $L^1_{\mathrm{loc}}$ for the trace measure of $T_0$. The case $k=2$ identifies
\[
\nu=\tfrac12\ddc(h^2T_0)-h \mu.
\]
Thus the local definitions of $\nu$ are independent of the representatives and determine a global measure. It is zero outside $C$, since the local differences are zero there.

We next prove the integration-by-parts formula with a bounded psh weight and verify its independence of the coordinate charts. Fix $k\ge1$ and set $f_l=\psi_l h_l^k$. For a real smooth function $g$ compactly supported in the chart, the bounded Bedford--Taylor integration-by-parts formula gives
\begin{align}
\int g f_l\,\ddc\phi\wedge T_0
={}&\int g\phi\,\ddc f_l\wedge T_0
+\int \phi f_l\,\ddc g\wedge T_0\notag\\
&+\int\phi\bigl(dg\wedge d^cf_l+df_l\wedge d^cg\bigr)\wedge T_0.
\label{eq:localgreen}
\end{align}
Here $f_l$ is a polynomial in bounded psh functions. Thus the formula follows by multilinearity from bounded Bedford--Taylor calculus relative to the closed positive current $T_0$. In particular, it applies to bounded psh $\phi$ and does not require differentiating this weight.

The measure $\ddc\phi\wedge T_0$ vanishes on pluripolar sets by Lemma~\ref{lem:mixedpolar}. Since $f_l\to h^k$ off the poles and $|f_l|\le M^k$, dominated convergence applies to the left side of \eqref{eq:localgreen} and to the term with $\ddc g$. Expanding $\ddc f_l\wedge T_0$, the terms with $\ddc\psi_l$ tend to zero in total variation on compact sets by \eqref{eq:localmass}. The cross terms with $d\psi_l$ tend to zero in total variation by Cauchy--Schwarz, \eqref{eq:gapenergy}, and \eqref{eq:cutoffenergy}. The remaining terms converge in total variation by \eqref{eq:cutoffmeasure} and by the increasing construction of $\nu$. Hence they may also be multiplied by the bounded Borel function $g\phi$.

We give the corresponding estimate for the first-order terms. Choose a compact set $K$ containing the support of $g$ in the chart and put
\[
E_j=K\cap\{u+v\le-j\}.
\]
For $l\ge j$, we have $h_l=h_j$ on $K\setminus E_j$. Plurifine locality of the energy measures, applied to the bounded difference $h_l-h_j$, gives
\[
\ind_{K\setminus E_j}d(h_l-h_j)\wedge d^c(h_l-h_j)\wedge T_0=0.
\]
By Cauchy--Schwarz, $dh_l\wedge d^cg\wedge T_0$ and $dg\wedge d^ch_l\wedge T_0$ therefore agree on $K\setminus E_j$ with the corresponding measures having index $j$. This is the local energy argument used in \cite[Section 2]{Vu22}. On $E_j$, Cauchy--Schwarz gives
\begin{align}
&\left\|\ind_{E_j}\psi_l h_l^{k-1}
 dh_l\wedge d^cg\wedge T_0\right\|\notag\\
&\qquad\le M^{k-1}
\left(\int_K dh_l\wedge d^ch_l\wedge T_0\right)^{1/2}
\left(\int_{E_j}dg\wedge d^cg\wedge T_0\right)^{1/2}
\longrightarrow0\quad(j\to\infty),\label{eq:firstordertail}
\end{align}
uniformly in $l$. The first factor is bounded by \eqref{eq:gapenergy}, using a cutoff equal to one on $K$. The second tends to zero because $E_j$ decreases to a pluripolar set and the trace measure of $T_0$ has no mass there. The same estimate applies with $dg\wedge d^ch_l$. On $K\setminus E_j$, we have $0\le1-\psi_l\le j/l$ for $l\ge j$, so $\psi_l\to1$ uniformly there, while $h_l=h_j$ and the first-order measures just considered are independent of $l$. Finally, the terms $h_l^k d\psi_l\wedge T_0$ paired with $dg$ tend to zero in total variation by \eqref{eq:cutoffenergy}. These observations prove total variation convergence of both first-order measures in \eqref{eq:localgreen}, also after multiplication by $\phi$.

The resulting first-order currents are $d(h^kT_0)$ and $d^c(h^kT_0)$: this follows by distributional convergence from $f_lT_0\to h^kT_0$. They are therefore independent of the representatives and agree on overlaps. We have proved the local identity
\begin{align}
\int g h^k\,\ddc\phi\wedge T_0
={}&\int g\phi\,\ddc(h^kT_0)
+\int\phi h^k\,\ddc g\wedge T_0\notag\\
&+\int\phi\bigl(dg\wedge d^c(h^kT_0)
+d(h^kT_0)\wedge d^cg\bigr).
\label{eq:greenlimit}
\end{align}
The measures and currents involving $h$ are supported in $C$. Choose smooth partition functions $g_1,\ldots,g_N$ supported in the charts such that $\sum_{j=1}^{N}g_j=\psi$, where $\psi$ is a smooth compactly supported function equal to one near $C$. Sum \eqref{eq:greenlimit} over these charts. The terms with derivatives of $g_j$ become the corresponding terms with derivatives of $\psi$, and vanish because $d\psi=\ddc\psi=0$ near $C$. We obtain
\[
\int h^k\,\ddc\phi\wedge T_0
=\int\psi\phi\,\ddc(h^kT_0).
\]
For $k=1$, use \eqref{eq:localchain1}; for $k\ge2$, use \eqref{eq:localchainq}. This proves \eqref{eq:ibp1}--\eqref{eq:ibpq}. Compact support and the preceding local bounds prove finiteness of all the integrals.
\end{proof}

\begin{proposition}\label{prop:compact}
Let $X$ be a complex manifold. Let $u\le v$ be families of local $\mathcal E$-potentials with the same transition functions, equal outside $C\cc X$. Let $\phi_1,\ldots,\phi_n$ be psh on a neighborhood $\Omega$ of $C$, with $-1\le \phi_l\le0$. If $\psi\in C_c^\infty(\Omega)$ is one near $C$, then
\begin{equation}
\frac1{n!}\int_\Omega(v-u)^n\ddc \phi_1\wedge\cdots\wedge\ddc \phi_n
\le\int_\Omega\psi(-\phi_1)(T^n-S^n),\label{eq:compact}
\end{equation}
where $T=\ddc u_j$, $S=\ddc v_j$. The signed integral is finite and independent of $\psi$.
\end{proposition}
\begin{proof}
Suppose first that $0\le h=v-u\le M$. For a mixed current $T_0$ of bidegree $(n-1,n-1)$ and $k\ge2$, Lemma~\ref{lem:ibp} gives
\[
\int h^k\ddc \phi\wedge T_0
\le k\int(-\phi)h^{k-1}(T-S)\wedge T_0
\le k\int h^{k-1}T\wedge T_0.
\]
Iterating with $\phi_n,\ldots,\phi_2$ yields
\[
\int h^n\ddc \phi_1\wedge\cdots\wedge\ddc \phi_n
\le n!\int h\ddc \phi_1\wedge T^{n-1}.
\]
For $n=1$ this step is equality. Put $\Sigma=\sum_{l=0}^{n-1}T^l\wedge S^{n-1-l}$. By \eqref{eq:ibp1} and multilinearity,
\[
\int\psi(-\phi_1)(T^n-S^n)
=\int h\ddc \phi_1\wedge\Sigma
\ge\int h\ddc \phi_1\wedge T^{n-1}.
\]
This proves \eqref{eq:compact} for bounded differences.

For the general case, set $u_j^{(t)}=\max(u_j,v_j-t)$, $t>0$. These functions belong to $\Eloc$, have the same transition functions as $v$, and equal $v_j$ outside $C$. Moreover,
\[
u_j^{(t)}\downarrow u_j,\qquad
0\le v_j-u_j^{(t)}\le t.
\]
Apply the inequality proved above for bounded differences to $u^{(t)},v$. Off the pole sets,
\[
v_j-u_j^{(t)}=\min\{v_j-u_j,t\}\uparrow v_j-u_j.
\]
The measure $\ddc \phi_1\wedge\cdots\wedge\ddc \phi_n$ vanishes on pluripolar sets, so the left side converges monotonically. Lemma~\ref{lem:weighted}, with a fixed smooth cutoff, gives convergence on the right. This proves \eqref{eq:compact}. Finally, $T^n-S^n$ is supported in $C$ by locality, which proves the assertions about the signed integral.
\end{proof}

\section{Proofs of the comparison and domination theorems}
For the functions in Definition~\ref{def:barrier}, write
\[
S_k=\ddc g_j^k,\qquad\mu=T^n,\qquad\nu=S^n,\qquad
\nu_0=R^n,\qquad\nu_k=S_k^n.
\]
Since $u_j\le g_j^k\le\max(u_j,v_j)$, Lemma~\ref{lem:classes} gives $g_j^k\in\Eloc$. By \eqref{eq:glue}, the currents $S_k$ agree on overlaps. The set
\[
A_k=\{u_j<v_j+h_j^k\}=\{u_j<g_j^k\}
\]
is independent of the chart, increases with $k$, and is contained in $C_k$.

The upper regularization of the increasing psh limit of $h_j^k$ is zero, because the limit is zero almost everywhere. An increasing locally bounded-above sequence of psh functions agrees with its regularized limit outside a pluripolar set \cite[Theorem 7.1]{BT82}. Taking the union over the countable cover and adjoining the pole sets of $u$ and $v$, we obtain a locally pluripolar set $E_0$ outside which all relevant potentials are finite and $h_j^k\uparrow0$. If $x\in A\setminus E_0$, then $v_j(x)-u_j(x)>0$, so $h_j^k(x)>u_j(x)-v_j(x)$ for all sufficiently large $k$. If $x\notin A$, the inequality $h_j^k\le0$ implies $x\notin A_k$. Therefore
\begin{equation}
A_k\uparrow A\text{ on }X\setminus E_0,\qquad
(g^k-u)\uparrow(v-u)_+\text{ there}.\label{eq:increase}
\end{equation}

We claim that
\begin{equation}
(\nu_k)_{\np}\ge\ind_{X\setminus A_k}\mu_{\np}
+\ind_{A_k}\nu_{\np}.\label{eq:nonpolar}
\end{equation}
Fix $j,k$ and work on a smaller hyperconvex coordinate neighborhood in $\Omega_j$, normalizing $u_j,v_j$ by a common constant. Set
\[
h_l=\max(h_j^k,-l),\qquad
 g_l=\max(u_j,v_j+h_l).
\]
Since $h_l$ is bounded psh, $v_j+h_l\in\mathcal E$. The functions $g_l\in\mathcal E$ decrease to $g_j^k\in\mathcal E$ and have the common minorant $u_j$. Moreover,
\[
A_k\subset\{u_j<v_j+h_l\}
\]
on this neighborhood. Plurifine locality and positivity of mixed products give the measure inequalities
\[
(\ddc g_l)^n
\ge\ind_{A_k}\bigl(\ddc(v_j+h_l)\bigr)^n
\ge\ind_{A_k}(\ddc v_j)^n.
\]
Testing against nonnegative continuous functions with compact support and using decreasing convergence yields
\[
\nu_k\ge\ind_{A_k}\nu.
\]
Since $v_j$ is finite on $A_k$, the measure on the right vanishes on pluripolar sets. Thus $(\nu_k)_{\np}\ge\ind_{A_k}\nu_{\np}$. On the complement of $A_k$, apply Lemma~\ref{lem:max} to $u_j,g_j^k$. Their maximum is $g_j^k$, and $\{u_j\ge g_j^k\}=X\setminus A_k$ on the chart, so
\[
\ind_{X\setminus A_k}(\nu_k)_{\np}
\ge\ind_{X\setminus A_k}\mu_{\np}.
\]
Combining the inequalities on these two disjoint sets proves \eqref{eq:nonpolar} locally and hence globally. The comparison on pluripolar sets in (b) gives
\begin{equation}
\mu_{\pp}\ge(\nu_k)_{\pp}\ge(\nu_0)_{\pp}.\label{eq:polar}
\end{equation}

\begin{proof}[Proof of Theorem~\ref{th:comparison}]
Set $h=-\phi_1\ge0$ and
\[
I_k=\frac1{n!}\int_X(g^k-u)^n\mu_\phi.
\]
Proposition~\ref{prop:compact} gives
\begin{equation}
I_k\le\int_{D_k}h(\mu-\nu_k).\label{eq:Ij}
\end{equation}
To interpret this expression, choose a smooth cutoff supported in $D_k$ and equal to one near $C_k$. The signed measure $\mu-\nu_k$ is supported in $C_k$, so the cutoff integral equals the displayed integral.

All measures have finite mass on $D_k$. Combining \eqref{eq:nonpolar}--\eqref{eq:Ij} gives
\[
I_k+\int_{A_k}h\nu_{\np}+\int_{D_k}h(\nu_0)_{\pp}
\le\int_{A_k}h\mu_{\np}+\int_{D_k}h\mu_{\pp}.
\]
The measure $(\nu_0)_{\pp}$ is carried by $E$, and $\mu_{\pp}$ is carried by $\{u=-\infty\}\subset A\cup E$. Since $v$ is finite on $A$, $\nu|_A=\nu_{\np}|_A$. Moreover,
\[
\int_{A_k}h\mu_{\np}+\int_{D_k}h\mu_{\pp}
\le\int_A h\mu+\int_{D_k\cap E}h\mu.
\]
On $E$, both $\nu_0$ and $\mu$ equal their polar restrictions. Consequently,
\[
I_k+\int_{A_k}h\nu+\int_{D_k\cap E}h\nu_0
\le\int_A h\mu+\int_{D_k\cap E}h\mu.
\]
If the right side of \eqref{eq:full} is infinite, the inequality holds in the extended sense. Otherwise $\int_A h\mu$ and $\int_E h\mu$ are finite. By \eqref{eq:increase}, monotone convergence applies to $I_k$ and to the $\nu$-integrals because $\mu_\phi$ and $\nu|_A$ vanish on pluripolar sets. The sets $D_k\cap E$ increase to $E$, so monotone convergence applies to those terms as well. Passing to the limit in the last inequality proves \eqref{eq:full}.

If $u\preceq_{\Cap}v$ locally, Lemma~\ref{lem:max} gives $(\nu_k)_{\pp}=\mu_{\pp}$. Cancel these finite polar measures on $D_k$ before taking limits. Then
\[
I_k+\int_{A_k}h\nu_{\np}\le\int_{A_k}h\mu_{\np}.
\]
The comparison in (c) also gives $\mu_{\pp}\le\nu_{\pp}$, hence $\mu_{\pp}|_A=0$. Monotone convergence now proves \eqref{eq:clean}.
\end{proof}

\begin{lemma}\label{lem:positive}
If $\mu\le\nu$ on $A$ and $\mu_{\pp}=(\nu_0)_{\pp}$, then $\nu_k\ge\mu$ on $X$ for every $k$.
\end{lemma}
\begin{proof}
Since $A_k\subset A$, \eqref{eq:nonpolar} gives $(\nu_k)_{\np}\ge\mu_{\np}$. On the polar parts, \eqref{eq:polar} and the assumed equality imply $(\nu_k)_{\pp}=\mu_{\pp}$. Adding the two inequalities gives the assertion, including the common pole set.
\end{proof}

\begin{lemma}\label{lem:zeta}
Suppose that $\rho$ satisfies \eqref{eq:rho}. Then
\[
\zeta=e^\rho\in\PSH(X)\cap L^\infty_{\mathrm{loc}}(X),
\qquad (\ddc\zeta)^n\ge g\,dV,
\]
where $e^{-\infty}=0$ and $g\in L^1_{\mathrm{loc}}(X,dV)$ is positive almost everywhere. Thus \eqref{eq:rho} is equivalent to the existence of a locally bounded psh function with this property.
\end{lemma}
\begin{proof}
Composition with the exponential gives a psh function $\zeta\ge0$, locally bounded above because $\rho$ is locally bounded above. For $l\ge1$, put
\[
\rho_l=\max(\rho,-l),\qquad \zeta_l=e^{\rho_l}.
\]
The function $t\mapsto e^t-e^{-l}t$ on $[-l,\infty)$, extended constantly to $(-\infty,-l)$, is convex and nondecreasing. Hence $\zeta_l-e^{-l}\rho_l$ is psh. Positivity and multilinearity of locally bounded Bedford--Taylor products give
\[
(\ddc\zeta_l)^n\ge e^{-nl}(\ddc\rho_l)^n.
\]
On $\{\rho>-l\}$, locality for bounded products and for maxima in $\mathcal E$ yields
\[
\begin{aligned}
\ind_{\{\rho>-l\}}(\ddc\zeta)^n
&=\ind_{\{\rho>-l\}}(\ddc\zeta_l)^n\\
&\ge e^{-nl}\ind_{\{\rho>-l\}}(\ddc\rho)^n
\ge e^{-nl}\ind_{\{\rho>-l\}}f\,dV.
\end{aligned}
\]
To replace $(\ddc\rho_l)^n$ by $(\ddc\rho)^n$ on $\{\rho>-l\}$, apply \cite[Theorem 4.1]{NP09} on smaller hyperconvex coordinate neighborhoods after a common constant shift. Let $g$ be the density of the absolutely continuous part of $(\ddc\zeta)^n$ with respect to $dV$. The displayed inequality gives $g\ge e^{-nl}f$ almost everywhere on $\{\rho>-l\}$. These sets increase to $\{\rho>-\infty\}$, whose complement has volume zero. Thus $g>0$ almost everywhere. The converse follows because locally bounded psh functions belong to $\Eloc$.
\end{proof}

\begin{lemma}\label{lem:compact-dom}
Suppose that $X$ satisfies \eqref{eq:rho}. Let $u\le v$ be families of local $\mathcal E$-potentials with the same transition functions, equal outside a compact set $K\cc X$. If $T^n\le S^n$, where $T=\ddc u_j$ and $S=\ddc v_j$, then $u=v$.
\end{lemma}
\begin{proof}
The sets $\{u_j=-\infty\}$ define a locally pluripolar set $E$ on $X$. Since $u_j\le v_j$, both potentials are finite outside $E$. Their difference defines a nonnegative Borel function $h=v_j-u_j$ on $X\setminus E$. Set $h=0$ on $E$. Then $h=0$ outside $K$. Monge-Amp\`ere measures of locally bounded psh functions vanish on $E$. Thus replacing the extended difference in Proposition~\ref{prop:compact} by $h$ leaves its left-hand side unchanged; its right-hand side retains the full measures $T^n$ and $S^n$.

Let $\zeta=e^\rho$ and $g>0$ almost everywhere be as in Lemma~\ref{lem:zeta}. Choose an open set $D\cc X$ containing $K$ and a constant $M>\max\{1,\sup_D\zeta\}$. Then
\[
\phi=\frac{\zeta}{M}-1\in\PSH(D),\qquad -1\le\phi\le0.
\]
Choose $\psi\in C_c^\infty(D)$ with $0\le\psi\le1$ and $\psi=1$ near $K$. Proposition~\ref{prop:compact}, with all test functions equal to $\phi$, gives
\[
0\le\frac1{n!M^n}\int_D h^n(\ddc\zeta)^n
\le\int_D\psi(-\phi)(T^n-S^n)\le0.
\]
The signed integral is finite because its weight is bounded and compactly supported. Since $(\ddc\zeta)^n\ge g\,dV$ and $g>0$ almost everywhere, we obtain $h=0$ almost everywhere on $D$. The set $E$ has volume zero, and the potentials agree outside $K$. Hence $u_j=v_j$ almost everywhere on each $\Omega_j$, and equality of their psh representatives gives $u_j=v_j$ everywhere.
\end{proof}

\begin{proof}[Proof of Theorem~\ref{th:dom}]
Fix $k$. The functions $g_j^k=\max(u_j,v_j+h_j^k)$ have the same transition functions as $u_j$ and equal $u_j$ outside $C_k$. They belong to $\Eloc$ by Lemma~\ref{lem:classes}, since they are psh and satisfy $g_j^k\ge u_j$. Lemma~\ref{lem:positive} gives $\nu_k\ge\mu$. Applying Lemma~\ref{lem:compact-dom} to $u$ and $g^k$ yields $g_j^k=u_j$ on every $\Omega_j$. Consequently,
\[
u_j\ge v_j+h_j^k\quad\text{on }\Omega_j.
\]
Letting $k\to\infty$ gives $u_j\ge v_j$ almost everywhere, since $h_j^k\uparrow0$ almost everywhere. The psh functions $\max(u_j,v_j)$ and $u_j$ therefore agree almost everywhere. Equality of their psh representatives gives $u_j\ge v_j$ everywhere, including at poles. Under the local assumption $u\preceq_{\Cap}v$, Lemma~\ref{lem:max} applied to $h=\max(u_j,v_j)$ yields \eqref{eq:balance}.
\end{proof}

\begin{proof}[Proof of Theorem~\ref{th:unique}]
Assuming $(R^n)_{\pp}=\mu_{\pp}$, apply Theorem~\ref{th:dom} with the functions $h_j^k$ for $(u,v)$ to get $u\ge v$. Interchanging $u,v$ and using the functions for $(v,u)$ gives $v\ge u$.

For the capacity alternative, on a neighborhood where $u\preceq_{\Cap}v$, Lemma~\ref{lem:max} gives $(R^n)_{\pp}=(T^n)_{\pp}$. On a neighborhood with the reverse order it gives $(R^n)_{\pp}=(S^n)_{\pp}$. Since $T^n=S^n=\mu$, these local identities agree and give $(R^n)_{\pp}=\mu_{\pp}$ globally. The first part applies.
\end{proof}

\begin{corollary}\label{cor:zero}
Suppose that $X$ satisfies \eqref{eq:rho}. Let $u,v$ be families of local $\mathcal E$-potentials with the same transition functions such that $u\le v$, $u=v$ outside a compact set, and $(\ddc u)^n=0$. Then $u=v$.
\end{corollary}
\begin{proof}
Since $0=(\ddc u)^n\le(\ddc v)^n$, Lemma~\ref{lem:compact-dom} applies.
\end{proof}

Recall the definition of Cegrell's class $\mathcal N(\Omega)$ \cite{Cegrell08}. Let $(D_k)$ be a fundamental exhaustion of a bounded hyperconvex domain $\Omega$ by smoothly bounded strictly pseudoconvex domains, with $D_k\cc D_{k+1}$ and $\bigcup_{k=1}^{\infty}D_k=\Omega$. For $h\in\mathcal E(\Omega)$, set
\[
h_k=\left(\sup\{\psi\in\PSH^-(\Omega):\psi\le h
\text{ on }\Omega\setminus D_k\}\right)^*,\qquad
\widetilde h=\left(\lim_{k\to\infty}h_k\right)^*.
\]
Here the star denotes upper semicontinuous regularization. Then
\[
\mathcal N(\Omega)=\{h\in\mathcal E(\Omega):\widetilde h=0\}.
\]
The following proposition gives a sufficient condition for Definition~\ref{def:barrier}.
\begin{proposition}\label{prop:N}
Let $\Omega\subset\mathbb C^n$ be a bounded hyperconvex domain, and let $u=(u_j)$ and $v=(v_j)$ be families of local $\mathcal E$-potentials on a common open cover $(\Omega_j)$ of $\Omega$, with the same pluriharmonic transition functions. Suppose that $h\in\mathcal N(\Omega)$ and that there is a compact set $K\Subset\Omega$, independent of $j$, such that $u_j\ge v_j+h|_{\Omega_j}$ on $\Omega_j\setminus K$ for every $j$. Then $(u,v)$ satisfies the condition in Definition~\ref{def:barrier}.
\end{proposition}
\begin{proof}
Let $h_k$ be the functions used in the definition of $\widetilde h$. The defining properties of $\mathcal N$ give $h\le h_k\le0$, $h_k\in\mathcal E$, and $h_k\uparrow0$ outside a pluripolar set; see \cite{Cegrell08}. Outside $\overline{D_k}$ the envelope equals $h$: locally all competitors are bounded above by the upper semicontinuous function $h$, and $h$ itself is admissible. Set $h_j^k=h_k|_{\Omega_j}$. The inequality in \eqref{eq:boundary} holds outside $K\cup\overline{D_k}$; equation \eqref{eq:glue} follows because $h_k$ is global and $u,v$ have the same transition functions.
\end{proof}

\begin{example}\label{ex:blowup}
Let $X$ be the blowup of the unit ball $B\subset\mathbb C^2$ at the origin:
\[
X=\{(w,[\xi_1:\xi_2])\in B\times\mathbb P^1:
w_1\xi_2=w_2\xi_1\}.
\]
Write $\pi:X\to B$ and $\tau:X\to\mathbb P^1$ for the projections, and $E=\pi^{-1}(0)$ for the exceptional curve. The function $\rho=|\pi|^2$ is smooth and psh. In a blowup chart where $\pi(z_1,z_2)=(z_1,z_1z_2)$, we have
\[
\rho=|z_1|^2(1+|z_2|^2),\qquad
\det\left(\frac{\partial^2\rho}{\partial z_j\partial\bar z_k}\right)=|z_1|^2.
\]
Thus $(\ddc\rho)^2$ has positive density off $E$, and \eqref{eq:rho} holds because $E$ has volume zero. However, $X$ has no global strictly psh function, since its restriction to $E\simeq\mathbb P^1$ would be strictly subharmonic on a compact curve.

Let $\omega_{\mathrm{FS}}$ be a Fubini--Study form on $\mathbb P^1$. For $\epsilon>0$, the form
\[
T=\ddc\rho+\epsilon\tau^*\omega_{\mathrm{FS}}
\]
is K\"ahler: it is the restriction of a product K\"ahler form under the embedding $X\hookrightarrow B\times\mathbb P^1$. It has no global psh potential, since
\[
\int_E T=\epsilon\int_{\mathbb P^1}\omega_{\mathrm{FS}}>0.
\]
Indeed, a potential of the smooth form $T$ would be smooth locally up to a pluriharmonic function, and Stokes' theorem on $E$ would give zero. Choose smooth local psh potentials $u_j$ of $T$. For a nonzero real $\chi\in C_c^\infty(X)$ and sufficiently small $|\delta|>0$, the form $S=T+\delta\ddc\chi$ is positive, with local potentials $v_j=u_j+\delta\chi$. These families have the same pluriharmonic transition functions and agree outside $\operatorname{supp}\chi$. Since $\int_E S=\int_E T>0$, the current $S$ also has no global psh potential. Taking $h_j^k=0$ verifies Definition~\ref{def:barrier} in both directions. Theorem~\ref{th:comparison} therefore applies to this pair of currents without global psh potentials; one may take $\phi_1=\phi_2=\rho-1$.
\end{example}

Some geometric hypothesis is necessary for domination and uniqueness. On $X=\mathbb P^1\times\mathbb C$, the functions
\[
u(x,z)=|z|^2,\qquad v(x,z)=\max\{|z|^2,1\}
\]
are locally bounded psh, satisfy $u\le v$, and agree outside the compact set $\mathbb P^1\times\{|z|\le1\}$. Both Monge-Amp\`ere measures vanish, since these functions are pulled back from the one-dimensional factor, but $u\ne v$. Every locally bounded global psh function on this product is constant on each $\mathbb P^1$ fiber and is therefore pulled back from $\mathbb C$, so its Monge-Amp\`ere square vanishes. By Lemma~\ref{lem:zeta}, this manifold cannot satisfy \eqref{eq:rho}.

\section{An application to maximal plurisubharmonic functions}\label{sec:CD}

We use the following equivalent definition of maximality: $u\in\PSH(\Omega)$ is maximal if every $v\in\PSH(\Omega)$ with $v\ge u$ and $v=u$ outside a compact subset of $\Omega$ satisfies $v=u$. The restriction of a maximal psh function to an open subset is maximal: a competing function equal to $u$ outside a compact subset extends by $u$ to the original domain.

On domains in $\mathbb C^2$, B\l ocki's results give
\begin{equation}
\Eloc=\PSH\cap W^{1,2}_{\mathrm{loc}},\qquad
f\in\Eloc:\quad f\text{ maximal}\ \Longleftrightarrow\ (\ddc f)^2=0;
\label{eq:blocki}
\end{equation}
see \cite[Theorems 1.1, 4.1 and Proposition 2.2]{Blocki04}, with local normalization for $\mathcal E$.

\begin{lemma}\label{lem:CDresidual}
Let $u$ be psh on a domain in $\mathbb C^2$ and satisfy \eqref{eq:CD} locally. Its Siu divisorial part $S_u$ is a locally finite effective real divisor. If $h_j$ are local psh potentials of $S_u$, then
\[
u=h_j+f_j,\qquad f_j\in\Eloc.
\]
\end{lemma}
\begin{proof}
First, a function $\psi\in\Eloc$ has zero generic Lelong number along every divisor. Otherwise, near a smooth point of such a divisor, the Siu decomposition gives
\[
\psi=\alpha\log|z_1|+v,\qquad \alpha>0,
\]
with $v$ psh. On a smaller neighborhood, $v\le C$, so $\psi-C\le\alpha\log|z_1|$. Lemma~\ref{lem:classes} and \eqref{eq:blocki} would imply $\log|z_1|\in W^{1,2}_{\mathrm{loc}}$, contrary to the divergence of $\int_0^\epsilon dt/t$ in the transverse variable.

Consider $\phi=\alpha\log\|f\|$ near a point $z_0$. Since the local ring of a smooth surface is factorial, there is a neighborhood of $z_0$ on which $f=h_0g$, where the components of $g$ have no common divisorial factor. On a sufficiently small neighborhood, $g$ has no zero or has $z_0$ as its only zero. In the latter case the local \L ojasiewicz inequality gives
\[
\alpha\log\|g(z)\|\ge \alpha N\log|z-z_0|-C.
\]
Since $\log|z-z_0|\in W^{1,2}_{\mathrm{loc}}(\mathbb C^2)$, Lemma~\ref{lem:classes} gives $\psi_0=\alpha\log\|g\|\in\Eloc$. Thus
\[
\phi=h+\psi_0,\qquad h=\alpha\log|h_0|,
\]
where $\ddc h$ is an effective divisor. If $\phi=0$, take $h=\psi_0=0$.

For each irreducible divisor $Z$, the two inequalities in \eqref{eq:CD} and monotonicity of Lelong numbers give
\[
\nu_Z(\phi)+\nu_Z(\psi)\ge\nu_Z(u),\qquad
\nu_Z(u)+\nu_Z(\psi)\ge\nu_Z(\phi).
\]
The first paragraph shows that $\nu_Z(\psi)=\nu_Z(\psi_0)=0$. Hence $\nu_Z(u)=\nu_Z(\phi)$, and the divisorial part of $\ddc u$ is precisely $\ddc h$ locally. This also proves local finiteness; see \cite{Siu74} for the Siu decomposition.

The positivity $\ddc u-S_u\ge0$ gives a psh representative $f_j$ of the distribution $u-h_j$, defined using the locally integrable functions $u$ and $h_j$. Equality almost everywhere gives $u=h_j+f_j$ everywhere. The distribution $h-h_j$ has a smooth pluriharmonic representative $\psi_j$. On a smaller neighborhood, $\psi_j\ge-C$ for some constant $C$. Outside the pole sets of $h$ and $h_j$, the lower inequality in \eqref{eq:CD} gives
\[
f_j\ge\psi_0+\psi+\psi_j\ge\psi_0+\psi-C.
\]
The inequality $f_j\ge\psi_0+\psi-C$ therefore holds almost everywhere, and hence everywhere because both sides are psh. The right side belongs to $\Eloc$, so Lemma~\ref{lem:classes} gives $f_j\in\Eloc$.
\end{proof}

\begin{lemma}\label{lem:divcompetitor}
Let $D\Subset\Omega\subset\mathbb C^2$, let $u\in\PSH(\Omega)$ satisfy \eqref{eq:CD} locally, and suppose $v\in\PSH(D)$ satisfies $v\ge u$ and $v=u$ on $D\setminus K$, where $K\Subset D$. Then the divisorial parts of $\ddc v$ and $\ddc u|_D$ coincide. In particular, $\ddc v\ge S_u|_D$.
\end{lemma}
\begin{proof}
Let $Z$ be an irreducible component of $S_u|_D$. Then $Z$ is a closed analytic subset of $D$ of pure dimension one. If $Z\subset K$, it would be compact, contradicting the absence of compact positive-dimensional analytic subsets in a domain in $\mathbb C^2$. Thus $Z\setminus K$ is a nonempty relatively open subset of $Z$. On its smooth locus choose a point generic for the divisorial Lelong numbers of both currents. Such a point exists because the exceptional sets are countable unions of proper analytic subsets of $Z$. Equality of $u$ and $v$ near this point yields $\nu_Z(v)=\nu_Z(u)$. The inequality $v\ge u$ also gives $\nu_Y(v)\le\nu_Y(u)$ for every divisor $Y$, so $\ddc v$ has no additional divisorial component. The assertion follows from the Siu decomposition.
\end{proof}

\begin{proof}[Proof of Theorem~\ref{th:CD}]
Lemma~\ref{lem:CDresidual} gives $u=h_j+f_j$ with $f_j\in\Eloc(D_j)$ and $\ddc f_j=T_u|_{D_j}$. On an overlap, the distribution $h_l-h_j$ has a smooth pluriharmonic representative $\psi_{jl}$, and
\[
f_j=f_l+\psi_{jl}.
\]
Lemma~\ref{lem:classes} therefore shows that $(\ddc f_j)^2=(\ddc f_l)^2$ on overlaps. These positive locally finite measures define the global measure in \eqref{eq:residualMA}. The same argument on a common refinement proves independence of the cover and of the chosen potentials of $S_u$.

\emph{(i)$\Rightarrow$(ii).} The restriction of a maximal psh function to an open subset is maximal, by the definition recalled above.

\emph{(ii)$\Rightarrow$(iii).} Fix $j$ and a point of $D_j$. Choose a ball $D_0\Subset D_j$ containing that point on which $u$ is maximal. If $g\in\PSH(D_0)$ satisfies $g\ge f_j$ and $g=f_j$ outside a compact subset of $D_0$, then $h_j+g$ is a psh compact perturbation of $u=h_j+f_j$ from above on $D_0$. Maximality of $u$ gives
\[
h_j+g=h_j+f_j\quad\text{on }D_0.
\]
Since $h_j$ is finite almost everywhere, $g=f_j$ almost everywhere and hence everywhere. Thus $f_j$ is maximal on $D_0$. By \eqref{eq:blocki}, $(\ddc f_j)^2=0$ on $D_0$. Such balls cover every $D_j$, so $T_u^2=0$ on $\Omega$.

\emph{(iii)$\Rightarrow$(i).} Let $v\in\PSH(\Omega)$ satisfy $v\ge u$ and $v=u$ outside $K\Subset\Omega$. Choose a relatively compact domain $D\Subset\Omega$ containing $K$, and choose local psh potentials $h_j$ of $S_u|_D$ on coordinate domains $D_j$. By Lemma~\ref{lem:CDresidual} and (iii),
\[
u=h_j+f_j,\qquad f_j\in\Eloc,\qquad(\ddc f_j)^2=0.
\]
Lemma~\ref{lem:divcompetitor} shows that $v$ has the same divisorial part as $u$ on $D$. Consequently, the distributions $v-h_j$ have psh representatives $g_j$, with $v=h_j+g_j$. The inequalities $g_j\ge f_j$ hold first almost everywhere and then everywhere; Lemma~\ref{lem:classes} gives $g_j\in\Eloc$. Similarly, $g_j=f_j$ on $D_j\setminus K$. On overlaps,
\[
f_j=f_l+\psi_{jl},\qquad g_j=g_l+\psi_{jl},
\]
where $\psi_{jl}$ is the smooth pluriharmonic representative of $h_l-h_j$. Thus $f,g$ are families of local $\mathcal E$-potentials with the same transition functions, $f\le g$, and $f=g$ outside the common compact set $K$. The current represented by $\ddc f_j$ has zero square. The function $\rho(z)=|z|^2$ satisfies \eqref{eq:rho} on $D$. Hence Corollary~\ref{cor:zero} gives $f=g$. Consequently $v=u$ on $D$, and hence on $\Omega$. The definition of maximality proves that $u$ is maximal.
\end{proof}

Dimension two is used in Lemma~\ref{lem:CDresidual} to prove that the functions $f_j$ belong to $\Eloc$. In the factorization $f=h_0g$, the zero set of $g$ is discrete, so $\alpha\log\|g\|$ has at most isolated singularities and belongs to $\Eloc$. The comparison argument is valid in every dimension.

\medskip
\noindent\textbf{Acknowledgements.} This work was completed while the first author was visiting the Department of Mathematics at the National University of Singapore as a Visiting Senior Research Fellow under the Singapore Academies Southeast Asia Fellowship (SASEAF) Programme, within the project ``Math, Sobolev Space'' (Grant No. E-146-00-0039-01). The first author would like to express his sincere gratitude to the Department of Mathematics at NUS for its warm hospitality and excellent working environment. He also gratefully acknowledges the Singapore National Academy of Science (SNAS) and the National Research Foundation (NRF), Singapore, for their financial support through the SASEAF Programme.

\end{document}